\documentclass{amsart}
\usepackage{tikz,hyperref,amsbsy}
\usepackage{amssymb}
\usepackage{pgfplots}
\usepackage{xcolor}
\pgfplotsset{compat=1.18}
\definecolor{myblue}{HTML}{1F77B4}
\usepackage[font=small,labelfont=bf]{caption}
\usepackage{enumerate}
\usepackage{enumitem}
\newtheorem{theorem}{Theorem}[section]

\newtheorem{lemma}[theorem]{Lemma}

\theoremstyle{definition}
\newtheorem{example}[theorem]{Example}
\numberwithin{equation}{section}
\newtheorem{remark}[theorem]{Remark}

\usepackage{mathrsfs}
\usepackage[bottom]{footmisc}

\newcommand{\bu}{\mathbf{u}}
\newcommand{\bv}{\mathbf{v}}
\newcommand{\bw}{\mathbf{w}}

\newcommand{\U}{\mathcal{U}} %sets of double-bases
\newcommand{\UU}{\overline{\mathcal{U}}} %sets of double-bases
\newcommand{\V}{\mathcal{V}} %sets of double-bases
\newcommand{\qtq}[1]{\quad \text{#1}\quad }%\quad\text{and}
\newcommand{\bj}[1]{\left(#1\right)}

\newcommand{\norm}[1]{\left\lVert #1\right\rVert}
\newcommand{\set}[1]{\left\lbrace #1\right\rbrace}%set

\title[Invariant measure for double base expansions]{invariant measure for double base expansions}
\author[W. Huang]{Wenduo Huang}
\address[W. Huang \& Y. Zou]{College of Mathematics and Statistics,
Shenzhen University,
Shenzhen 518060,
People's Republic of China}
\email{wdhuang2001@163.com, yuruzou@szu.edu.cn}
\author[V. Komornik]{Vilmos Komornik}
\address[V. Komornik]{Département de mathématique,
Université de Strasbourg,
7 rue René Descartes,
67084 Strasbourg Cedex, France}
\email{vilmos.komornik@math.unistra.fr}
\author[Y. Zou]{Yuru Zou}

\keywords{$(q_0,q_1)$-expansions, dynamical properties, greedy and lazy maps, unique invariant probability measure, univoque set, Lebesgue measure}

\subjclass{28D05, 11A63}

\thanks{ }

\begin{document}

\begin{abstract}
Given a pair $Q=(q_0,q_1)\in(1,\infty)^2$ with $q_0+q_1\ge q_0q_1$, a sequence $(c_i)\in\set{0,1}^\infty$ is called a $Q$-expansion of $x$ if
\begin{equation*}
x=\sum_{i=1}^{\infty}\frac{c_i}{q_{c_1}\cdots q_{c_i}}.
\end{equation*}
We primarily study  the dynamical properties of  the greedy and lazy maps, which are the piecewise-linear maps on the interval $I_Q=[0,\,1/(q_1-1)]$ defined by the corresponding algorithms for $Q$-expansions. 
We show that the greedy and lazy maps each of which has a unique absolutely continuous invariant probability measure, equivalent to the Lebesgue measure on the intervals
\begin{equation*}
\left[0,\frac{q_0}{q_1}\right)\qtq{and}\left(\frac{q_1}{q_0(q_1-1)}-1,\frac{1}{q_1-1}\right],
\end{equation*}
respectively. 
Furthermore, the corresponding dynamical systems are exact on $I_Q$. 
As a dynamical consequence, under the stronger condition  $q_0+q_1>q_0q_1$ the set of points having unique $Q$-expansions has Lebesgue measure zero, and almost every $x\in I_{Q}$ admits a continuum of $Q$-expansions.
\end{abstract}

\maketitle

\section{Introduction}\label{s1}

The study of \emph{$q$-expansions} of the form
\begin{equation*}
x=\sum_{i=1}^{\infty}\frac{d_i}{q^i},\quad (d_i)\in\set{0,1}^{\infty}
\end{equation*}
where the \emph{base} $q\in (1,2]$ is a given real number, goes back to the seminal works of Rényi \cite{Renyi1957} and Parry \cite{Parry1960}. 
Since then a large body of results has revealed deep connections with number theory, Diophantine approximation, symbolic dynamics, ergodic theory, and fractal geometry (see, e.g., \cite{AdamczewskiBugeaud2006, Kalle2020, Charlier2023, Masai2018, deVries2016, deVries2022, Sidorov2009}). 
It is well known that $x$ admits a $q$-expansion if and only if $x$ lies in the compact interval $[0,\frac{1}{q-1}]$ \cite{Renyi1957}, and that in case $q\in(1,2)$ (in Lebesgue's sense) almost every $x$ has $2^{\aleph_0}$  different $q$-expansions \cite{Sidorov2003}.
Furthermore, for any prescribed $k\in\mathbb N\cup\set{\aleph_0,2^{\aleph_0}}$ there exist $2^{\aleph_0}$ bases $q\in(1,2)$ in which $x=1$ has exactly $k$ different $q$-expansions \cite{Erdos1991,Erdos1992,Erdos1994}. 
The theory has subsequently been generalized to multiple-base representations \cite{Neunhauserer2021,Komornik2022,Li2021}, leading to a much richer symbolic structure.

Hochman and Shmerkin \cite{HochmanShmerkin2015} propose a pioneering fractal-geometric condition for determining whether a measure  is pointwise normal in base $n$. 
This condition is $C^1$-stable and generalizes relevant results from integer bases to the more general setting of Pisot bases $\beta>1$. 
Furthermore, Huang and Wang \cite{HuangWang2025} proved that for two distinct non-integer bases $\beta_1,\beta_2>1$, the corresponding Rényi--Parry measures coincide if and only if $\beta_1$ is a root of $x^{2}-qx-p=0$ with $p,q\in\mathbb{N}$ and $p\le q$, and moreover $\beta_{2}=\beta_{1}+1$.

Let $T_\beta:[0,1)\to[0,1)$ be the $\beta$-transformation $T_\beta(x)=\{\beta x\}$.
Rényi developed a canonical-digits method: for each finite canonical word $\omega=\omega_1\cdots\omega_n$ he defines the corresponding cylinder $\mathscr{E}(\omega)\subset(0,1)$ and proves that, for every $n\in\mathbb{N}$, the family
$\{\mathscr{E}(\omega):|\omega|=n\}$ is pairwise disjoint and forms a cover of $(0,1)$.
This cylinder decomposition yields two-sided estimates for the measures of multiple preimages, providing uniform upper and lower bounds for $\mu(T_\beta^{-n}E)$.
Invoking a Dunford--Miller type limit theorem, Rényi deduces the existence of a $T_\beta$-invariant measure $\nu$ which is equivalent to Lebesgue ($\nu\ll \lambda$), together with the almost everywhere convergence of Birkhoff averages:
\begin{equation*}
\frac1n\sum_{k=0}^{n-1} g\!\left(T_\beta^k x\right)\longrightarrow
\int g\,\mathrm{d}\nu=\int g\,h_{\beta}\,\mathrm{d}\lambda ,
\end{equation*}
where $h_{\beta}=\frac{\mathrm{d}\nu}{\mathrm{d}\lambda}$.
Finally, by an application of \emph{Knopp's theorem}, the system $(T_\beta,\nu)$ is ergodic. 
Parry showed that, for each $\beta>1$, the unique $T_\beta$-invariant measure admits an explicit step density:
\begin{equation*}
h_\beta(x)
=  \sum_{n=0}^{\infty} \beta^{-n}\,\mathbf{1}_{[0,\,T_\beta^{\,n}(1))}(x).
\end{equation*}
In the random setting, Dajani and de Vries \cite{DajaniDeVries2005,DajaniDeVries2007} introduce a probabilistic switching
between the greedy and the lazy $\beta$-maps to form the skew product $K_\beta$.
They prove that the topological entropy satisfies
$h_{\mathrm{top}}(K_\beta)=\log\!\big(1+\lfloor\beta\rfloor\big)$ and that there
is a unique measure of maximal entropy $\nu_\beta$. 
In parallel, working on the space $BV$ of functions of bounded variation, they establish Lasota--Yorke type estimates and, via Ces\`aro averaging, construct and uniquely identify an absolutely continuous invariant probability measure that is equivalent to the Lebesgue measure. 
Equivalently, its density $f\in BV$ is a fixed point of the corresponding Frobenius--Perron operator: $\mathcal Pf=f$, and the associated measure solves the affine fixed-point equation
\begin{equation*}
\mu \;=\; p\,\mu\circ T_\beta^{-1} \;+\; (1-p)\,\mu\circ S_\beta^{-1},
\end{equation*}
which is consistent with the invariance of the skew product. 
These results give a systematic description of the measure-theoretic structure and the coding features of random $\beta$-expansions.

Recently, following the work \cite{Komornik2022}, the papers \cite{Hu2025}, \cite{Komornik2025} and \cite{Lu2026} focused on \emph{double-base expansions} or \emph{$Q$-expansions} of the form
\begin{equation*}
x=\pi_Q(c):=\sum_{i=1}^{\infty}\frac{c_i}{q_{c_1}\cdots q_{c_i}},\quad c=(c_i)\in\set{0,1}^\infty,
\end{equation*}
associated with a \emph{double-base} $Q=(q_0,q_1)\in(1,\infty)^2$.
We have $\pi_Q(c)\in I_Q:=[0,\frac{1}{q_1-1}]$ for all $c\in\{0,1\}^\infty$.
It turned out that the numbers
\begin{equation*}
\ell_Q:=\frac{q_1}{q_0(q_1-1)}-1\qtq{and}r_Q:=\frac{q_0}{q_1}
\end{equation*}
play an important role, and they have expansions if and only if $q_0+q_1\ge q_0q_1$.
Henceforth always assuming this condition, it was proved in \cite{Hu2025} that all three sets
\begin{align*}
&\U:=\set{Q\in (1,\infty)^2 : \text{both }\ell_Q\text{ and } r_Q \text{ have  unique $Q$-expansions}},\\
&\UU \ \text{ denotes the topological closure of }\U,\\
&\V:=\set{Q\in (1,\infty)^2 : \text{both }\ell_Q\text{ and } r_Q \text{ have  unique doubly infinite $Q$-expansions}}
\end{align*}
have Hausdorff dimension two.
(The set $\V$ is always closed, $\U\subsetneqq\UU\subsetneqq\V$, and the three sets reduce to the usual ones in the equal-base case $q_0=q_1$.)
Based on this work, the paper \cite{Komornik2025} provided a complete description of the topological relationships and properties among the following variants of $\U$, $\UU$ and $\V$:
\begin{align*}
&\U_Q:=\set{x\in I_Q : x\text{ has a unique $Q$-expansion}},\\
&\UU_Q\ \text{ denotes the topological closure of }\U_Q,\\
&\V_Q:=\set{x\in I_Q : x\text{ has at most one doubly infinite $Q$-expansion}}.
\end{align*}
Here too, $\V_Q$ is always closed, and $\U_Q\subset\UU_Q\subset\V_Q$, but the inclusions are not always strict: four different cases may occur according to the value of $Q$.
Furthermore, the points in $\mathcal{V}_Q \setminus \mathcal{U}_Q$ may have 2, 3 or $\aleph_0$  distinct expansions (in the equal base case we never have three expansions). The paper \cite{Lu2026} gave a explicit formula for the Hausdorff dimension of $\U_Q$ and showed the continuity of this formula as a function of $Q$.

Henceforth we always assume that
\begin{equation}\label{Q-assumption}
Q=(q_0,q_1)\in(1,\infty)^2\qtq{with}q_0+q_1\ge q_0q_1.
\end{equation}
Then every number $x\in I_Q$ has at least one \emph{$Q$-expansion}, namely the \emph{greedy $Q$-expansion} $b(x,Q)=(b_i)=(b_1,b_2,\ldots)$.
It is defined recursively by the greedy algorithm, choosing in each step the largest possible digit, so that it is the lexicographically largest expansion of $x$.
The greedy greedy map $G(x): I_Q\mapsto I_Q$, defined by the formula
\begin{equation}\label{greedy-map}
G(x)=
\begin{cases}
    G_0(x)=q_0x,& x\in I_0:=[0,\frac{1}{q_1})\\
    G_1(x)=q_1x-1,& x\in I_1:=[\frac{1}{q_1},\frac{1}{q_1-1}]
\end{cases}
\end{equation}
generates the greedy expansion in the sense that 
\begin{equation*}
b_n=j\Longleftrightarrow G^{n-1}(x)\in I_j,\quad j=0,1.
\end{equation*}

Similarly, each $x\in I_Q$ has a lexicographically smallest expansion $l(x,Q)=(l_i)=(l_1,l_2,\ldots)$, also  called the  \emph{lazy $Q$-expansion} of $x$.
It is generated by the lazy map $L(x):I_Q\mapsto I_Q$, defined by the formula
\begin{equation}\label{lazy-map}
L(x)=
\begin{cases}
    L_0(x)=q_0x, & x\in J_0:=[0,\frac{1}{q_0(q_1-1)}]\\
    L_1(x)=q_1x-1, & x\in J_1:=(\frac{1}{q_0(q_1-1)},\frac{1}{q_1-1}]
\end{cases}
\end{equation}
in the sense that
\begin{equation*}
l_n=j\Longleftrightarrow L^{n-1}(x)\in J_j,\quad j=0,1.
\end{equation*}

The following theorem establishes the existence of absolutely continuous invariant probability measures for $G$ and 
$L$.
We recall that a dynamical system $(X,\Sigma,\mu,T)$ is called \emph{exact} if the set
\begin{equation*}
\cap_{n=1}^{\infty}\set{T^{-n}(A)\ :\ A\in\Sigma}
\end{equation*}
is trivial, i.e., it contains only sets of measure zero and their complements.

\begin{theorem}\label{t11}
Assume that $Q=(q_0,q_1)\in(1,\infty)^2$ with $q_0+q_1\ge q_0q_1$.

\begin{enumerate}
\item[(i)] There exist unique probability measures $\mu_{g,Q}$ and $\mu_{l,Q}$ on $I_Q$ that are invariant under $G$ and $L$, respectively.

\item[(ii)] They are absolutely continuous with respect to the Lebesgue measure $m$, and  and their density functions $h_{g,Q}$ and $h_{l,Q}$ are bounded. Moreover, $h_{g,Q}$ is non-increasing, while $h_{l,Q}$ is non-decreasing.

\item[(iii)] Finally, the dynamical systems $(I_Q,\Sigma_{I_Q},\mu_{g,Q},G)$ and $(I_Q,\Sigma_{I_Q},\mu_{l,Q},L)$ are exact.
\end{enumerate}
\end{theorem}

We recall the classical hierarchy of statistical properties:
\begin{equation*}
\text{exact} \ \Longrightarrow\ \text{strong mixing}\ \Longrightarrow\ \text{ergodic};
\end{equation*}
see, e.g., the discussion in \cite{MackeyTyranKaminskaWalther2016}.
Therefore, by Theorem~\ref{t11}, the systems
\begin{equation*}
(I_Q,\Sigma_{I_Q},\mu_{g,Q},G)\qtq{and}(I_Q,\Sigma_{I_Q},\mu_{l,Q},L)
\end{equation*}
are strong mixing and ergodic.

Next we determine the supports of the measures $\mu_{g,Q}$ and $\mu_{l,Q}$.

\begin{theorem}\label{t12}\ 
Assume that $Q=(q_0,q_1)\in(1,\infty)^2$ with $q_0+q_1\ge q_0q_1$.

\begin{enumerate}
\item[(i)] The invariant  measure $\mu_{g,Q}$ is  equivalent to the Lebesgue measure $m$ on the interval $[0,r_Q)$, and vanishes outside it.

\item[(ii)] Similarly, the invariant  measure $\mu_{l,Q}$ is  equivalent to the Lebesgue measure $m$ on the interval $\left(\ell_Q,\frac{1}{q_1-1}\right]$, and vanishes outside it.
\end{enumerate}
\end{theorem}

Next we give an explicit expression for the invariant density functions.
Let us introduce the jump functions
\begin{align}
\tilde{h}_{g,Q}(x)
&=\sum_{n=0}^{\infty}\frac{1}{q_1^{s(n)}q_0^{n-s(n)}}1_{[0,G^n(r_Q)}(x)\label{tilde-h-greedy}
\intertext{and}
\tilde{h}_{l,Q}(x)
&=\sum_{n=0}^{\infty}\frac{1}{q_0^{n-t(n)}q_1^{t(n)}}1_{(L^n(\ell_Q),\frac{1}{q_1-1}]}(x),\label{tilde-h-lazy}
\end{align}
where $s(n)$ denotes the sum of the first $n$ digits of the greedy expansion of $r_Q$, and $t(n)$ denotes the sum of the first $n$ digits of the lazy expansion of $\ell_Q$.
The following result states that the invariant densities $h_{g,Q}$ and $h_{l,Q}$ are obtained from the functions \eqref{tilde-h-greedy} and \eqref{tilde-h-lazy}.

\begin{theorem}\label{t13}
Assume that $Q=(q_0,q_1)\in(1,\infty)^2$ with $q_0+q_1\ge q_0q_1$.
Then the invariant densities $h_{g,Q}$ and $h_{l,Q}$ are proportional to $\tilde{h}_{g,Q}$ and $\tilde{h}_{l,Q}$, respectively.
\end{theorem}

In the single base case, the univoque set $\mathcal{U}_Q$ is Lebesgue null set \cite{DajaniDeVries2007}. 
The same phenomenon persists in the double base setting:
 
\begin{theorem}\label{t14}
If $Q=(q_0,q_1)\in(1,\infty)^2$ with $q_0+q_1>q_0q_1$, then $\mathcal{U}_Q$ is a Lebesgue null set.
\end{theorem}

The following theorem is an immediate consequence of Theorem~\ref{t14}.

\begin{theorem}\label{t15}
If $Q=(q_0,q_1)\in(1,\infty)^2$ with $q_0+q_1>q_0q_1$, then almost every $x\in I_Q$ has a continuum of $Q$-expansions.
\end{theorem}

Our work is based on the important papers of Lasota and Yorke \cite{LasotaYorke1982,LasotaYorke1973}.
We introduce the Frobenius--Perron transfer operators associated with the greedy and lazy maps in the double-base
setting, denoted by $\mathcal{G}$ and $\mathcal{L}$. 
Our approach to the existence of absolutely continuous invariant probability measures is based on a linear conjugacy which rescales the state space $I_Q=[0,1/(q_1-1)]$ to $[0,1]$.
After conjugation, the resulting maps are piecewise $C^2$ expanding on $[0,1]$,
so that classical results of Lasota--Yorke yield invariant densities. 
Pushing these measures forward through the conjugacy produces invariant probability measures for $G$ and $L$ on $I_Q$, which are absolutely continuous with respect to Lebesgue measure. Moreover, by invariance of exactness under measurable
conjugacy, the corresponding dynamical systems are exact. 
In the spirit of Parry \cite{Parry1960}, we further construct explicit step representations of the associated (non-normalized) invariant densities, after normalization these give
the probability densities of the two absolutely continuous invariant probability measures.

\section{Proof of Theorem~\ref{t11}} \label{s2}

Our proofs are based on a classical theorem of Lasota and Yorke \cite{LasotaYorke1982} that we recall for the convenience of the reader.

Let $T$ be a doubly measurable transformation of the unit interval $[0, 1]$ into itself, satisfying the following conditions:
\begin{enumerate}
\item[(a)] there is a partition $0=a_0<\cdots<a_N=1$ such that for each integer $k=1,\ldots ,N$ the restriction of $T$ to the interval $[a_{k-1},a_k)$ is continuous and convex,

\item[(b)] $T_{a_{k-1}}=0$ and $T'_{a_{k-1}}>0$ for $k=1,\ldots ,N$, and

\item[(c)] $T'(0)>1$.
\end{enumerate}
Let $P$ denote the Frobenius--Perron operator associated with $T$, defined by the formula
\begin{equation*}
\int_APf\, dx=\int_{T^{-1}(A)}f\, dx
\end{equation*}
for all $f\in L^1([0,1])$.
Note that $P$ is linear, preserves the integral, and is contractive: $\norm{Pf}_{L^1}\le\norm{f}_{L^1}$ for all $f\in L^1([0,1])$.

Under these assumptions the following theorem holds:

\begin{theorem}[\cite{LasotaYorke1982}, Theorem 4]\label{t21}\  
\begin{enumerate}
\item[(i)] There exists a unique normalized absolutely continuous measure $\mu_g$ that is invariant under $T$.

\item[(ii)] The system $\bj{[0,1],\mu_g,T}$ is exact,  and the density $g=d\mu_g/dx$ is bounded and non-increasing.

\item[(iii)] Moreover, $\norm{P^nf-g}_1\to 0$ as $n\to\infty$  for every nonnegative function $f\in L^1([0,1])$ of integral one (with respect to the Lebesgue measure).
\end{enumerate}
\end{theorem}

\begin{remark}\label{r22}
It follows from Theorem \ref{t21} and the continuity of $P$ that the density function is a fixed point of $P$:
\begin{equation*}
Pg=P(\lim_nP^nf)=\lim_nP^{n+1}f=g.
\end{equation*}

Conversely, if $Ph=h$ for some nonnegative function $h\in L^1([0,1])$ with $\norm{h}_{L^1}>0$, then $h$ is a constant multiple of $g$ because
\begin{equation*}
g=\lim_nP^n\bj{\norm{h}_{L^1}^{-1}h}=\norm{h}_{L^1}^{-1}\lim_nP^nh
=\norm{h}_{L^1}^{-1}\lim_nh=\norm{h}_{L^1}^{-1}h.
\end{equation*}
\end{remark}

\begin{proof}[Proof of Theorem~\ref{t11}]
Let us introduce the linear bijection $\phi:[0,1]\to I_Q$ given by the formula $\phi(u):=u/(q_1-1)$, and the conjugate map
\begin{equation*}
\widetilde G:=\phi^{-1}\circ G\circ \phi:[0,1]\to[0,1].
\end{equation*}
A direct computation gives
\begin{equation*}
\widetilde G(u)=
\begin{cases}
q_0u, & u\in[0,a_1),\\[1mm]
q_1u-(q_1-1), & u\in[a_1,1],
\end{cases}
\end{equation*}
where $a_1=(q_1-1)/q_1$. 

Take the partition $0=a_0<a_1<a_2=1$. For each integer $k=1,2$ the restriction of $\widetilde G$ to the interval $[a_{k-1},a_k)$ is linear, hence continuous and convex.
Moreover,
\begin{equation*}
\widetilde G(a_0)=\widetilde G(0)=0,\qquad \widetilde G(a_1)=q_1a_1-(q_1-1)=0,
\end{equation*}
and
\begin{equation*}
\widetilde G'(a_0)=q_0>1,\qquad \widetilde G'(a_1)=q_1>0.
\end{equation*}
By Theorem \ref{t21} there exists a $\widetilde G$-invariant probability
measure $\mu_{\widetilde G}$ on $[0,1]$ which is absolutely continuous with respect to Lebesgue measure $m_1$, and the system $([0,1],\Sigma_{[0,1]},\mu_{\widetilde G},\widetilde G)$ is exact.

Let $m$ be the Lebesgue measure on $I_Q$ and $d\mu_{\widetilde G}:=\widetilde f(u)\,dm_1$. 
The formula
\begin{equation*}
\mu_{g,Q}(E):=\mu_{\widetilde G}(\phi^{-1}E),\quad E\in\Sigma_{[0,1]}
\end{equation*}
defines a probability measure on $I_Q$ because $\mu_{g,Q}(I_Q)=\mu_{\widetilde G}([0,1])=1$.
Furthermore, $\mu_{g,Q}$ is $G$-invariant because
\begin{align*}
\mu_{g,Q}(G^{-1}E)
&=\mu_{\widetilde G}\bigl(\phi^{-1}(G^{-1}E)\bigr)
=\mu_{\widetilde G}\bigl((G\circ\phi)^{-1}(E)\bigr) \\
&=\mu_{\widetilde G}\bigl((\phi\circ\widetilde G)^{-1}(E)\bigr)
=\mu_{\widetilde G}\bigl(\widetilde G^{-1}(\phi^{-1}E)\bigr)
=\mu_{\widetilde G}(\phi^{-1}E)
=\mu_{g,Q}(E)
\end{align*}
for all $E\in\Sigma_{[0,1]}$, where we used the conjugacy relation $G\circ\phi=\phi\circ\widetilde G$ and the
$\widetilde G$-invariance of $\mu_{\widetilde G}$. 
By the change of variable $x=u/(q_1-1)$, we have
\begin{equation*}
\mu_{g,Q}(E)=\mu_{\widetilde G}(\phi^{-1}E)=\int_{\phi^{-1}E}\widetilde f(u)\,dm_1(u)=\int_E(q_1-1)\widetilde f\bigl((q_1-1)x\bigr)\,dm(x).
\end{equation*}
This shows that $\mu_{g,Q}$ is absolutely continuous with respect to the Lebesgue measure on $I_Q$.

Since exactness is invariant under measure-theoretical isomorphisms \cite{Rokhlin1967}, we conclude that
\begin{equation*}
([0,1],\Sigma_{[0,1]},\mu_{\widetilde G},\widetilde G)\ \text{is exact}
\quad\Longleftrightarrow\quad
(I_Q,\Sigma_{I_Q},\mu_{g,Q},G)\ \text{is exact}.
\end{equation*}
Therefore $(I_Q,\Sigma_{I_Q},\mu_{g,Q},G)$ is exact. 
This completes the proof for $G$.
The argument for $L$ is analogous.
\end{proof}

\begin{remark}\label{r23}
The proof of Theorem 1.1 is not specific to the double-base setting: it remains valid for all regular alphabet-base systems in the terminology of \cite{Komornik2022}.
%
%The proof of Theorem 1.1 is not specific to the double-base setting. 
%After an affine rescaling, the maps become piecewise affine expanding maps on the unit interval, so the same argument should extend to more general multiple-base settings.
\end{remark}

\section{Proof of Theorem~\ref{t12}} \label{s3}

We consider only the case of the greedy map; the proof for the lazy map is analogous.
For the proof first we need a special partition of $I_Q$ into a family of pairwise disjoint intervals. 
We recall that
\begin{equation*}
I_0=\left[0, \frac{1}{q_1}\right),\quad
I_1=\left[\frac{1}{q_1}, \frac{1}{q_1-1}\right],\quad
G^{-1}_0(x)=\frac{x}{q_0}\qtq{and}G^{-1}_1(x)=\frac{x+1}{q_1}.
\end{equation*} 
Starting with $I_0$ and $I_1$, we define recursively the intervals $I_{\bw}$ for all words
$\bw=w_1\cdots w_n\in\{0,1\}^n$ of length $n\geq 2$ by the formula
\begin{equation*}
I_{\bw}:=G^{-1}_{w_1\cdots w_{n-1}}(I_{w_n})\cap I_{w_1\cdots w_{n-1}}.
\end{equation*}        
Here $G^{-1}_{w_1\cdots w_{n-1}}(x)=G^{-1}_{w_1}\circ\cdots \circ G^{-1}_{w_{n-1}}(x)$.
For example,  $I_{00}=[0,\frac{1}{q_0q_1})$ and $I_{01}=[\frac{1}{q_0q_1},\frac{1}{q_1})$.

\begin{lemma}\label{l31}
For each $n=1,2,\ldots,$ the intervals $I_{\bw}$ for words of length $n$ form a disjoint partition of $I_Q$.
\end{lemma}

\begin{proof}
First we show that the intervals $I_{\bw}$ and $I_{\bv}$ for two different words $w_1\cdots w_n$ and $v_1\cdots v_n$ of length $n$ are  disjoint.

This clearly holds for $n=1$ by a direct inspection of $I_0$ and $I_1$.
Proceeding by induction, let $n\ge 2$, and assume that the property is true for $n-1$.

If $w_1\cdots w_{n-1}\ne v_1\cdots v_{n-1}$, then $I_{w_1\cdots w_{n-1}}$ and $I_{v_1\cdots v_{n-1}}$ are disjoint by the induction hypothesis.
Since $I_{\bw}\subset I_{w_1\cdots w_{n-1}}$ and $I_{\bv}\subset I_{v_1\cdots v_{n-1}}$ by construction, then the intervals $I_{\bw}$ and $I_{\bv}$ are also disjoint.

If $w_1\cdots w_{n-1}=v_1\cdots v_{n-1}$, then $w_n\ne v_n$.
Assume on the contrary that there exists a point $x\in  I_{\bw}\cap I_{\bv}$.  
Then we have
\begin{equation*}
x\in G^{-1}_{w_1\cdots w_{n-1}}(I_1)\cap G^{-1}_{w_1\cdots w_{n-1}}(I_0)
\end{equation*}
by the definition of $I_{\bw}$ and $I_{\bv}$.
This implies that $G_{w_{n-1}\cdots w_1}(x)\in I_0\cap I_1$, contradicting the disjointness of $I_0$ and $I_1$.
    
We show that
\begin{equation*}
I_Q=\cup_{\bw\in\{0,1\}^n}I_{\bw}
\end{equation*}
for every $n$.
This holds for $n=1$ because $I_Q=I_0\cup I_1$ by a direct inspection.
Proceeding by induction, if the equality holds for some $n\ge 1$, then it also holds for $n+1$ in place of $n$ because
\begin{equation*}
I_{w_1\cdots w_n0}\cup I_{w_1\cdots w_n1}=G^{-1}_{w_1\cdots w_n}(I_{Q})\cap I_{w_1\cdots w_n}=I_{w_1\cdots w_n}
\end{equation*}  
for every word $w_1\cdots w_n$ of length $n$.
\end{proof}

Next we study the intervals 
\begin{equation*}
J_{\bw}:=G_{w_{n}}\circ\cdots \circ G_{w_{1}}(I_{\bw}).
\end{equation*}

\begin{example}\label{e32}
We consider the intervals $I_{\bw}$ and $J_{\bw}$ for the words $\bw\in\{0,1\}^N$.

If $N =1$, then
\begin{equation*}
I_0=\left[0,\frac{1}{q_1}\right),\quad
I_1=\left[\frac{1}{q_1},\frac{1}{q_1-1}\right],\quad
J_{0}=\left[0,r_Q\right)\qtq{and}
J_{1}=\left[0,\frac{1}{q_1-1}\right].
\end{equation*}
If $N=2$, then
\begin{multline*}
I_{00}=\left[0,\frac{1}{q_0q_1}\right),\quad
I_{01}=\left[\frac{1}{q_0q_1},\frac{1}{q_1}\right),\quad
I_{10}=\left[\frac{1}{q_1},\frac{1}{q_1}+\frac{1}{q_1^2}\right),
I_{11}=\left[\frac{1}{q_1}+\frac{1}{q_1^2},\frac{1}{q_1-1}\right],
\end{multline*}
and
\begin{equation*}
J_{00}=[0,r_Q),\quad
J_{01}=[0,q_0-1),\quad
J_{10}=[0,r_Q),\quad
J_{11}=\left[0,\frac{1}{q_1-1}\right].
\end{equation*}
\end{example}

\begin{lemma}\label{l33}
Let $\bw=\bu w_{N}\in\{0,1\}^N$ with $\bu=w_1\cdots w_{N-1}$.

\begin{enumerate}
\item[(i)] If $\bw=1^N$ then $J_{\bw}=I_Q=[0,\frac{1}{q_1-1}]$. 

\item[(ii)] If $\bw\neq 1^N$ and $J_{\bw}\neq\emptyset$, then $J_{\bw}=[0,\xi_{\bw})$ for some $\xi_{\bw}\in(0,\frac{1}{q_1-1})$.
\end{enumerate} 
\end{lemma}

\begin{proof}
The property holds for $N=1$ because 
\begin{equation*}
J_0=G_0(I_0)=[0,r_Q)\text{ and }J_1=G_1(I_1)=I_Q.
\end{equation*}

Proceeding by induction, let $N\ge 2$, and assume that the statement holds for $N-1$. 
Since
\begin{equation*}
I_{\bu0}=G^{-1}_{w_1\cdots w_{N-1}}(I_0)\cap I_{w_1\cdots w_{N-1}},
\end{equation*}
we have
\begin{equation*}
J_{\bu0}=G_0\circ G_{w_{N-1}}\circ\cdots\circ G_{w_1}(I_{\bu0})=G_0(I_0\cap J_{\bu}).
\end{equation*}
Similarly, $J_{\bu1}=G_1(I_1\cap J_{\bu})$.

If $w_N=0$, then there are two subcases:
\begin{itemize}
    \item If $\bu=1^{N-1}$ , then \(J_{\bw}=J_{\bu0}=G_0(I_0\cap J_{\bu})= G_0(I_0)=[0,r_Q)\).
    \item If $\bu\neq 1^{N-1}$ and $J_{\bu}\neq\emptyset$, then $J_{\bw}=[0,\min\{r_Q,q_0\xi_{\bu}\})$.
\end{itemize}

If $w_N=1$, then there are again two subcases:
\begin{itemize}
\item If \(J_{\bu}\ne I_Q\), then the inductive hypothesis gives
\(J_{\bu}=[0,\xi_{\bu})\) with \(\xi_{\bu}<\frac{1}{q_1-1}\).
Then \(I_1\cap J_{\bu}=[1/q_1,\xi_{\bu})\) is right-open, hence \(J_{\bw}=[0,q_1 \xi_{\bu}-1)\) and $q_1 \xi_{\bu}-1<1/(q_1-1)$.

\item If \(J_{\bu}=I_Q\), then \(I_1\cap J_{\bu}=I_1\) and
\(J_{\bw}=G_1(I_1)=I_Q\). \qedhere
\end{itemize}
\end{proof}

\begin{lemma}\label{l34}
For every finite word $\bu\in\{0,1\}^*$ there exists an integer $m\ge 1$ such that $J_{u0^m}=[0,r_Q)$.
\end{lemma}

\begin{proof}
If $J_u=I_Q$, then
\begin{equation*}
J_{u0}=G_0(I_0\cap J_u)=G_0(I_0)=[0,r_Q),
\end{equation*}
so that the claim holds with $m=1$.

If $J_u\neq I_Q$, then $J_u=[0,\xi_u)$ with some $\xi_u\in (0,\frac{1}{q_1-1})$ by Lemma \ref{l33}.
We prove by induction that
\begin{equation}\label{31}
J_{u0^m}=[0,\min\{r_Q,q_0^m\xi_u\})
\end{equation}
for every $m\ge 0$. 
This will imply the lemma because $q_0^m\to\infty$ and therefore $q_0^m\xi_u\ge r_Q$ if $m\ge 1$ is large enough.

The relation \eqref{31} holds for $m=0$ by our assumption.
If it holds for some $m\ge 0$, then, since $G_0(x)=q_0x$ and $r_Q=q_0/q_1$, we have
\begin{equation*}
J_{u0^{m+1}}
=G_0(I_0\cap J_{u0^m})
=G_0\bj{\left[0,\min\set{\frac{1}{q_1},q_0^m\xi_u}\right)}
=[0,\min\{r_Q,q_0^{m+1}\xi_u\}),
\end{equation*}
which proves the induction step.
\end{proof}

It follows from the definitions \eqref{greedy-map}--\eqref{lazy-map} of the greedy and lazy maps $G$ and $L$ that the corresponding Frobenius–-Perron operators $\mathcal{G}$ and $\mathcal{L}$ are given by the following formulas:
\begin{align}
\mathcal{G}f(x)&=\frac{1}{q_0}f\bj{\frac{x}{q_0}}1_{[0,r_Q)}(x)+\frac{1}{q_1}f\bj{\frac{x+1}{q_1}}\label{FP greedy},\\
\mathcal{L}f(x)&=\frac{1}{q_0}f\bj{\frac{x}{q_0}}+\frac{1}{q_1}f\bj{\frac{x+1}{q_1}}1_{(\ell_Q,\frac{1}{q_1-1}]}(x).\label{FP lazy}
\end{align}

\begin{lemma}\label{l35}
The invariant  densities  $h_{g,Q}$ and $h_{l,Q}$ vanish outside the intervals $[0,r_Q)$ and  $\left(\ell_Q,\frac{1}{q_1-1}\right]$, respectively.
\end{lemma}

\begin{proof}
By symmetry we consider only $h_{g,Q}$.
Since $h_{g,Q}$ is a fixed point of the Frobenius--Perron operator $\mathcal{G}$ by Remark \ref{r22}, we infer from the relation \eqref{FP greedy} that
\begin{equation}\label{34}
h_{g,Q}(x)=\frac{1}{q_1}h_{g,Q}\bj{\frac{x+1}{q_1}}
\end{equation}
for almost every $x\in I:=[r_Q,\frac{1}{q_1-1}]$.
Starting from any $x_0\in I$, we may define a non-decreasing sequence $(x_n)$ in $I$ by the formula 
\begin{equation*}
x_{n+1}:=\frac{x_n+1}{q_1},\quad n=0,1,\ldots .
\end{equation*}
Indeed, if $x_n\in I$ for some $n\ge 0$, then 
\begin{equation*}
x_{n+1}-x_n=\frac{q_1-1}{q_1}\bj{\frac{1}{q_1-1}-x_n}\ge 0
\end{equation*}
and
\begin{equation*}
x_{n+1}\le \frac{\frac{1}{q_1-1}+1}{q_1}
=\frac{1+(q_1-1)}{(q_1-1)q_1}
=\frac{1}{q_1-1}.
\end{equation*}
We deduce from \eqref{34} by induction that
\begin{equation*}
h_{g,Q}(x_0)=\frac{1}{q_1^n}h_{g,Q}(x_n),\quad n=1,2,\ldots .
\end{equation*}
Since $q_1>1$ and the function $h_{g,Q}$ is bounded (up to a set of measure zero) by Theorem \ref{t11}, letting $n\to\infty$ we obtain that $h_{g,Q}(x_0)=0$ for a.e. $x_0\in I$.
\end{proof}

\begin{lemma}\label{l36}
The functions $h_{g,Q}$ and $h_{l,Q}$ are positive almost everywhere on $[0,r_Q)$ and $(\ell_Q,\frac{1}{q_1-1}]$, respectively. 
Moreover, there exist two positive constants $\delta_g,\delta_\ell>0$ such that
\begin{align}\label{35}
&m\big(\{x\in [0,r_Q):\; h_{g,Q}(x)<\delta_g\}\big)=0
\intertext{and}
&m\big(\{x\in \left(\ell_Q,\frac{1}{q_1-1}\right]:\;  h_{l,Q}(x)<\delta_\ell\}\big)=0.\notag
\end{align}
\end{lemma}

\begin{proof}
By symmetry we only prove the result for the greedy map.
Let us consider the following set:
\begin{equation*}
\mathcal{A}:=\{y\in [0,r_Q):\,h_{g,Q}(y)=0\}.
\end{equation*}
If $y\in \mathcal{A}$, then it follows from \eqref{FP greedy} that
\begin{equation}\label{36}
0=h_{g,Q}(y)=\mathcal{G}^N h_{g,Q}(y)=\sum_{\bw\in\{0,1\}^N}
\frac{1}{A_{\bw}}\,\mathbf 1_{J_{\bw}}(y)\,
h_{g,Q}\bigl(G^{-1}_{\bw}(y)\bigr)
\end{equation}
for every positive integer $N$.
Since $h_{g,Q}\ge 0$,  from \eqref{36} we have
$h_{g,Q}(G_{\bw}^{-1}(y))=0$ whenever $y\in J_{\bw}$. 
In particular, if $y\in \mathcal{A}$ and $I_{\bw}\subset [0,r_Q)=J_{\bw}$, then $G_{\bw}^{-1}(y)\in\mathcal{A}$.
Since $G_{\bw}^{-1}$ maps $J_{\bw}$ onto $I_{\bw}$, we conclude that
\begin{equation}\label{37}
G_{\bw}^{-1}(\mathcal{A})\subset \mathcal{A}\cap I_{\bw}\qtq{whenever}I_{\bw}\subset [0,r_Q)=J_{\bw}.
\end{equation}

We claim that $m(\mathcal{A})=0$.
Indeed, assume on the contrary that $m(\mathcal{A})>0$.
Then $m(\mathcal{A}\cap I)>0$ for every nonempty open interval $I\subset [0,r_Q)$.
To show this we choose a finite word $w$ satisfying $I_{\bw}\subset I$ and $J_{\bw}=[0,r_Q)$.
This is possible by Lemma \ref{l31}, \ref{l34} because the length of the intervals $I_{\bw}$ tends to zero as the length of the words $\bw$ tends to infinity.

Then,  since $G^{-1}_{\bw}$ is an affine bijection of $J_{\bw}=[0,r_Q)$ onto $I_{\bw}\subset I$, using \eqref{37} we obtain that
\begin{equation*}
m(\mathcal{A}\cap I)
\ge m(\mathcal{A}\cap I_{\bw})
\ge m\bj{G^{-1}_{\bw}(\mathcal{A})}
=\frac{1}{A_{\bw}}m(\mathcal{A})>0;
\end{equation*}
here we use the notation
\begin{equation*}
A_{\bw}:=\prod_{i=1}^n q_{w_i}.
\end{equation*}

Since $m(\mathcal{A}\cap I)>0$ for every nonempty open interval $I\subset [0,r_Q)$, $A$ is dense in $[0,r_Q)$, i.e., $h_{g,Q}$ vanishes on a dense subset of $[0,r_Q)$.
Since $h_{g,Q}$ is monotone, and therefore continuous a.e., hence we infer that $h_{g,Q}=0$ a.e. on $[0,r_Q)$.
Applying Lemma \ref{l35} we conclude that $h_{g,Q}=0$ a.e. on $I_Q$.
But this is impossible because $h_{g,Q}$ is a density function, and hence $\int_{I_Q} h_{g,Q}\,dm=1$.

We have proved that $m(\mathcal{A})=0$; hence $h_{g,Q}>0$ a.e. on $[0,r_Q)$.
Therefore there exists a number $\varepsilon_0>0$ such that the set
\begin{equation*}
\mathcal{B}:=\set{x\in [0,r_Q):h_{g,Q}(x)\ge \varepsilon_0}
\end{equation*}
has positive Lebesgue measure. 
Since $h_{g,Q}$ is non-increasing, $\mathcal{B}$ is a non-degenerate interval.
Therefore, similarly as above, we may choose a word $\bw$ such that $I_{\bw}\subseteq \mathcal{B}$ and $J_{\bw}=[0,r_Q)$. 
If $x\in [0,r_Q)$, then $G^{-1}_{\bw}(x)\in I_{\bw}\subseteq \mathcal{B}$.
Hence $h_{g,Q}(G^{-1}_{\bw}(x))\geq \varepsilon_0$, and therefore
\begin{equation*}
h_{g,Q}(x)=\mathcal{G}^{N}h_{g,Q}(x) \ge\ \frac{1}{A_{\bw}}\,h_{g,Q}(G^{-1}_{\bw}(x))
\ \ge\ \frac{\varepsilon_0}{A_{\bw}}>0.
\end{equation*}
This proves \eqref{35} with $\delta_g=\frac{\varepsilon_0}{A_{\bw}}$.
\end{proof}

\begin{proof}[Proof of Theorem~\ref{t12}]
By symmetry we consider only the case of the greedy map.
We already know from Theorem \ref{t11} and Lemma \ref{l35} that
 $\mu_{g,Q}$ is absolutely continuous with respect to the Lebesgue measure $m$, and vanishes outside $[0,r_Q)$.
For the equivalence on $[0,r_Q)$ it remains to observe that if $\mu_{g,Q}(E)=0$ for some measurable subset of $[0,r_Q)$, then 
\begin{equation*}
0=\int_E\, d\mu_{g,Q}=\int_Eh_{g,Q}\, dm,
\end{equation*}
and hence $m(E)=0$ because $h_{g,Q}>0$ a.e. on $E$ by Lemma \ref{l36}.
\end{proof}

\section{Proof of Theorem~\ref{t13}} \label{s4}

By Remark \ref{r22} it is sufficient to show that the functions 
$\tilde{h}_{g,Q}$ and $\tilde{h}_{l,Q}$, given by \eqref{tilde-h-greedy}--\eqref{tilde-h-lazy} are fixed points of the Perron--Frobenius operators $\mathcal{G}$ and $\mathcal{L}$ (see \eqref{FP greedy}--\eqref{FP lazy}), i.e.,
\begin{align}
\tilde{h}_{g,Q}(x)
&=\frac{1}{q_0}\tilde{h}_{g,Q}\bj{\frac{x}{q_0}}1_{[0,r_Q)}(x)
+\frac{1}{q_1}\tilde{h}_{g,Q}\bj{\frac{x+1}{q_1}}\label{41}
\intertext{and}
\tilde{h}_{l,Q}(x)
&=\frac{1}{q_0}\tilde{h}_{l,Q}\bj{\frac{x}{q_0}}
+\frac{1}{q_1}\tilde{h}_{l,Q}\bj{\frac{x+1}{q_1}}1_{(\ell_Q,\frac{1}{q_1-1}]}(x)\label{42}
\end{align}
for almost every $x\in I_Q$.
By symmetry we prove only the first equality.

\begin{proof}[Proof of \eqref{41}]
Setting
\begin{equation*}
f_n(x)=\frac{1_{[0,G^n(r_Q))}(x)}{q_1^{s(n)}q_0^{n-s(n)}},\quad n=0,1,\ldots,
\end{equation*}
where $(d_1,d_2,\ldots)$ is the greedy expansion of $r_Q$ and $s(n)=\sum_{i=1}^nd_i$, by the definition \eqref{tilde-h-greedy} we may rewrite $\tilde{h}_{g,Q}$ in the form
\begin{equation*}
\tilde{h}_{g,Q}=\sum_{n=0}^{\infty}f_n.
\end{equation*}
We have 
\begin{equation*}
\frac{1}{q_0}f_n\bj{\frac{x}{q_0}}1_{[0,r_Q)}(x)
=\frac{1_{[0,\; q_0G^n(r_Q))\cap [0,r_Q)}(x)}{q_1^{s(n)}q_0^{n+1-s(n)}}.
\end{equation*}
If $d_{n+1}=0$, then $G^n(r_Q)<\frac{1}{q_1}$ and $G^{n+1}(r_Q)=q_0G^n(r_Q)<r_Q$,  so that
\begin{equation*}
\frac{1}{q_0}f_n\bj{\frac{x}{q_0}}1_{[0,r_Q))}(x)
=\frac{1_{[0,\; G^{n+1}(r_Q)}(x)}{q_1^{s(n)}q_0^{n+1-s(n)}}=\frac{1_{[0,\; G^{n+1}(r_Q))}(x)}{q_1^{s(n+1)}q_0^{n+1-s(n+1)}}=f_{n+1}(x).
\end{equation*}
If $d_{n+1}=1$, then $G^n(r_Q)\geq \frac{1}{q_1}$ and $q_0G^n(r_Q)\geq r_Q$, whence 
\begin{equation*}
\frac{1}{q_0}f_n(x)\bj{\frac{x}{q_0}}1_{[0,r_Q)}(x)=\frac{1_{[0, r_Q)}(x)}{q_1^{s(n)}q_0^{n+1-s(n)}}.
\end{equation*}

Similarly, we have
\begin{equation*}
\frac{1}{q_1}f_{n}\bj{\frac{x+1}{q_1}}=\frac{1_{[0,\;G^n(r_Q))}\bj{\frac{x+1}{q_1}}}{q_1^{s(n)+1}q_0^{n-s(n)}}=\frac{1_{[0,\;q_1G^n(r_Q)-1))}(x)}{q_1^{s(n)+1}q_0^{n-s(n)}}.
\end{equation*}
If $d_{n+1}=0$, then $1_{[0,\;G^n(r_Q))}\bj{\frac{x+1}{q_1}}=0$. 
If $d_{n+1}=1$, then $G^n(r_Q)\geq \frac{1}{q_1}$ and $G^{n+1}(r_Q)=q_1G^n(r_Q)-1$, and therefore 
\begin{equation*}
\frac{1}{q_1}f_{n}\bj{\frac{x+1}{q_1}}=\frac{1_{[0,\;G^{n+1}(r_Q))}(x)}{q_1^{s(n+1)}q_0^{n+1-s(n+1)}}=f_{n+1}(x).
\end{equation*}

Combining the above four relations we obtain that
\begin{equation*}
\frac{1}{q_0}f_n\bj{\frac{x}{q_0}}1_{[0,r_Q)}(x)+\frac{1}{q_1}f_n\bj{\frac{x+1}{q_1}}=
\begin{cases}
f_{n+1}(x)&\text{ if } d_{n+1}=0,\\
f_{n+1}(x)+\frac{1_{[0, r_Q)}(x)}{q_1^{s(n)}q_0^{n+1-s(n)}}&\text{ if } d_{n+1}=1,
\end{cases}
\end{equation*}
and this may be rewritten in the form
\begin{equation*}
\frac{1}{q_0}f_n\bj{\frac{x}{q_0}}1_{[0,r_Q)}(x)+\frac{1}{q_1}f_n\bj{\frac{x+1}{q_1}}=f_{n+1}(x)+\frac{q_1}{q_0}\frac{d_{n+1}}{q_1^{s(n+1)}q_0^{n+1-s(n+1)}}1_{[0,r_Q)}(x).
\end{equation*}
Using the last formula, the required relation \eqref{41}
follows:
\begin{align*}
\frac{1}{q_0}\tilde{h}_{g,Q}\bj{\frac{x}{q_0}}&1_{[0,r_Q)}(x)
+\frac{1}{q_1}\tilde{h}_{g,Q}\bj{\frac{x+1}{q_1}}\\
&=\sum_{n=0}^{\infty}\bigl(f_{n+1}(x)+\frac{q_1}{q_0}\frac{d_{n+1}}{q_1^{s(n+1)}q_0^{n+1-s(n+1)}}1_{[0,r_Q)}(x)\bigr)\\
&=\sum_{n=1}^{\infty}\frac{1}{q_1^{s(n)}q_0^{n-s(n)}}1_{[0,G^n(r_Q))}(x)+\frac{q_1}{q_0}\sum_{n=1}^{\infty}\frac{d_n}{q_1^{s(n)}q_0^{n-s(n)}}1_{[0,r_Q)}(x)\\
&=\sum_{n=1}^{\infty}\frac{1}{q_1^{s(n)}q_0^{n-s(n)}}1_{[0,G^n(r_Q))}(x)+1_{[0,r_Q)}(x)\\
&=\tilde{h}_{g,Q}(x).\qedhere
\end{align*}
\end{proof}

\begin{example}\label{e41}
If \(q_{0}+q_{1}=q_{0}q_{1}\), then
\begin{equation*}
\frac{q_{0}}{q_{1}}=\frac{1}{q_{1}-1}
\qquad\text{and}\qquad
\frac{q_{1}}{q_{0}(q_{1}-1)}-1=0.
\end{equation*}
Hence the greedy and lazy expansions of \(\frac{q_{0}}{q_{1}}\) and
\(\frac{q_{1}}{q_{0}(q_{1}-1)}-1\) are \(1^{\infty}\) and \(0^{\infty}\), respectively.
Consequently,
\begin{equation*}
\tilde h_{g,Q}=\sum_{n=0}^{\infty}\frac{1}{q_{1}^{n}}=\frac{q_{1}}{q_{1}-1},
\qquad
\tilde h_{l,Q}=\sum_{n=0}^{\infty}\frac{1}{q_{0}^{n}}=\frac{q_{0}}{q_{0}-1}.
\end{equation*}
After normalization, the invariant probability densities are \(h_{g,Q}=h_{l,Q}=q_{1}-1\).
Hence
\begin{equation*}
\mu_{g,Q}(E)=\mu_{l,Q}(E)=\frac{m(E)}{m(I_Q)}
\end{equation*} 
for every measurable $E\subseteq I_Q$, so that
$\mu_{g,Q}$ and $\mu_{l,Q}$ coincide with the normalized Lebesgue measure on $I_Q$.
\end{example}

\begin{example}\label{e42}
We prescribe that \(q_{0}/q_{1}\) and \(q_{1}/(q_{0}(q_{1}-1))-1\) have
greedy and lazy expansions \(1110^{\infty}\) and \(001^{\infty}\), respectively.
Solving the resulting constraints yields \(q_{0}= 2.1479\) and \(q_{1}=1.46557\). Consequently,
\(h_{g,Q}\) is the jump function displayed in Figure~\ref{fig:hgq-step}
with breakpoints at \(1/q_{1}\), \(q_{0}-1\), and \(q_{0}/q_{1}\).

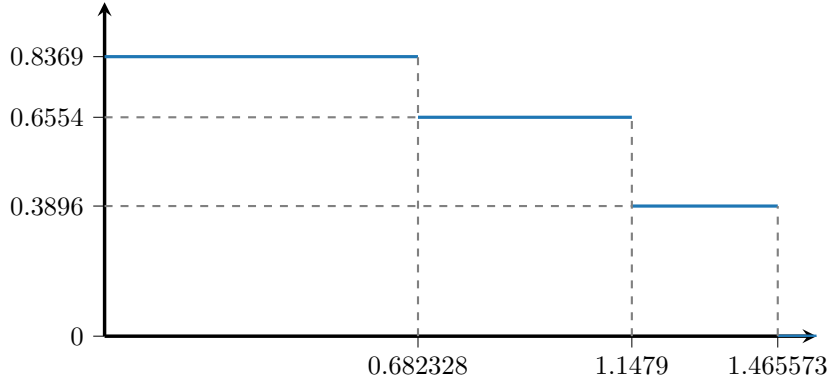
\begin{figure}[htbp]
  \centering
\begin{tikzpicture}
  % 常量
  \pgfmathsetmacro\bone{0.682328}
  \pgfmathsetmacro\btwo{1.1479}
  \pgfmathsetmacro\bthree{1.465573}

  \pgfmathsetmacro\vone{0.8369}
  \pgfmathsetmacro\vtwo{0.6554}
  \pgfmathsetmacro\vthree{0.3896}
  \pgfmathsetmacro\vfour{0.0}

  \pgfmathsetmacro\xmax{1.55}
  \pgfmathsetmacro\ymax{1}

  \begin{axis}[
    width=11cm, height=6cm,
    xmin=0, xmax=\xmax,
    ymin=0, ymax=\ymax,
    axis background/.style={fill=white},
    grid=none,
    axis x line=bottom,
    axis y line=left,
    axis line style={line width=1.3pt},
    every axis x line/.append style={stealth}, % 轴线箭头
    every axis y line/.append style={stealth},
    tick align=outside,
    tick style={black},
    xtick={\bone,\btwo,\bthree},
    xticklabels={0.682328, 1.1479, 1.465573},
    ytick={0,0.3896,\vtwo,\vone},
    yticklabels={0, 0.3896, 0.6554, 0.8369},
  ]

  % ---- 水平蓝色分段（不画竖向蓝线）----
  \addplot[very thick, color=myblue] coordinates {(0,\vone)   (\bone,\vone)};
  \addplot[very thick, color=myblue] coordinates {(\bone,\vtwo) (\btwo,\vtwo)};
  \addplot[very thick, color=myblue] coordinates {(\btwo,\vthree) (\bthree,\vthree)};
  \addplot[very thick, color=myblue] coordinates {(\bthree,\vfour) (\xmax,\vfour)};

  % ---- 跳变处虚线：竖线到 x 轴 & 水平虚线到断点 ----
  \addplot[gray, dashed, thick] coordinates {(\bone,0)   (\bone,\vone)};
  \addplot[gray, dashed, thick] coordinates {(\btwo,0)   (\btwo,\vtwo)};
  \addplot[gray, dashed, thick] coordinates {(\bthree,0) (\bthree,\vthree)};

  \addplot[gray, dashed, thick] coordinates {(0,\vtwo)   (\bone,\vtwo)};
  \addplot[gray, dashed, thick] coordinates {(0,\vthree) (\btwo,\vthree)};
\end{axis}
\end{tikzpicture}
\caption{The step function $h_{g,Q}$ with $q_0=2.1479$, $q_1=1.46557$.}
  \label{fig:hgq-step}
\end{figure}
\end{example}

%\blue{I have checked the paper until here.}

\section{Proof of Theorems~\ref{t14} and\ref{t15}}\label{s5}

For the proof of Theorem~\ref{t14} we need two lemmas.
The first one is a  classical integral inequality of Chebyshev: 

\begin{lemma}\label{l51}
Let $f,g\in L^1[a,b]$. 
If $f$ and $g$ have the same monotonicity on $[a,b]$, then
\begin{equation*}
\int_a^b f(x)g(x)\,dx
\;\ge\; \frac{1}{b-a}\Big(\int_a^b f(x)\,dx\Big)\Big(\int_a^b g(x)\,dx\Big).
\end{equation*}
If $f$ and $g$ have the opposite monotonicity on $[a,b]$, then the reverse inequality holds:
\begin{equation*}
\int_a^b f(x)g(x)\,dx
\;\le\; \frac{1}{b-a}\Big(\int_a^b f(x)\,dx\Big)\Big(\int_a^b g(x)\,dx\Big).
\end{equation*}
Moreover, equality holds if and only if $f$ or $g$ is constant almost everywhere.
\end{lemma}

\begin{proof}
See, e.g., \cite{MitrinovicPecaricVasic1993} and \cite{Jakubowski2021}.
\end{proof}

Our next lemma exhibits an important property of the greedy map $G$ (see \eqref{greedy-map}).
Let us introduce a partial inverse $H:I_Q\to I_Q$ of $G$ by the formula
\begin{equation*}
H(t):=\frac{1+t}{q_1}.
\end{equation*}
Observe that $H$ is an increasing affine map, $H(t)\ge t$ for all $t$ with equality only if $t=1/(q_1-1)$, and that $G(H(t))=t$ for all $t\in I_Q$.

\begin{lemma}\label{l52}
Except $x=1/(q_1-1)$, every point $x\in I_Q$ if of the form $x=H^k(t)$ for some $k\ge 0$ and $t\in [0,r_Q]$.
More precisely,
\begin{equation*}
\bigcup_{k=0}^{\infty}H^k\bj{[0,r_Q]}=\left[0,\frac{1}{q_1-1}\right).
\end{equation*}

%\begin{equation*}
%\cup_{n=1}^{\infty}H^n\bj{[q_0-1,r_Q]}=\left[r_Q,\frac{1}{q_1-1}\right).
%\end{equation*}
\end{lemma}

\begin{proof}
First we note that
\begin{equation*}
H\bj{[0,r_Q]}=H\bj{\left[0,\frac{q_0}{q_1}\right]}
=\left[\frac{1}{q_1},\frac{q_0+q_1}{q_1^2}\right].
\end{equation*}
Since the left endpoint of the interval on the right hand side belongs to the interval on the left hand side, iterating the  increasing affine map $H$, and using the property $H(t)\ge t$ we obtain that 
\begin{equation*}
\bigcup_{k=0}^{\infty}H^k\bj{[0,r_Q]}=[0,L)\qtq{with}L=\lim H^k(r_Q).
\end{equation*}
Since $H$ is continuous,
\begin{equation*}
H(L)=H\bj{\lim H^k(r_Q)}=\lim H^{k+1}(r_Q)=L,
\end{equation*}
and hence $L=1/(q_1-1)$.
\end{proof}

\begin{proof}[Proof of Theorem~\ref{t14}]
The normalized errors of the greedy expansion $b(x)=(b_i)$ of $x$ are defined by 
\begin{equation*}
\theta_n(b(x))=q_{b_1}q_{b_2}\cdots q_{b_{n}}\bj{x-\sum_{i=1}^{n}\frac{b_i}{q_{b_1}\cdots q_{b_i}}}.
\end{equation*}
It is easily seen that $G^n(x)=\theta_n(b(x))$.
Therefore a straightforward application of Birkhoff’s ergodic theorem yields that
\begin{equation}\label{51}
\lim_{n\to \infty}\frac{1}{n}\sum_{j=0}^{n-1}\theta_j(b(x))
=\lim_{n\to \infty}\frac{1}{n}\sum_{j=0}^{n-1}G^j(x)
=\int_{I_Q}xh_{g,Q}(x)\,dx
\end{equation}
for $\mu_{g,Q}$-almost every $x\in I_Q$, and hence for Lebesgue almost every $x\in [0,r_Q]$ (we apply Theorem \ref{t12} (i) here).

In order to show that \eqref{51} holds in fact for almost every $x\in I_Q$, let us denote by $N_0$ the set of numbers $x\in [0,r_Q]$ for which \eqref{51} does not hold.
We already know that $N_0$ is a Lebesgue null set.
Since the map $H$ of Lemma \ref{l52} is affine, the set
\begin{equation*}
N:=\bj{\bigcup_{k=0}^{\infty}H^k(N_0)}\bigcup \set{\frac{1}{q_1-1}}
\end{equation*} 
is also a Lebesgue null set.

If $x\in I_Q\setminus N$, then by Lemma \ref{l52} there exists a $k\ge 0$ such that $x\in H^k\bj{[0,r_Q]}$, and then $G^kx\in [0,r_Q]\setminus N_0$ (we recall that $G\circ H$ is the identity map).
Since the Cesàro convergence is not affected by the addition of a finite number of elements to a sequence, we conclude that $x$ also satisfies \eqref{51}.

Similarly to the relations \eqref{51}, by symmetry we also have
\begin{equation}\label{52}
\lim_{n\to \infty}\frac{1}{n}\sum_{j=0}^{n-1}\theta_j(l(x))=\int_{I_Q}xh_{l,Q}(x)\,dx
\end{equation}
for Lebesgue almost every $x\in I_Q$, where $l(x)$ denotes the lazy expansion of $x$.

Assume for a moment that the integrals in \eqref{51} and \eqref{52} are different.
Then
\begin{equation*}
\lim_{n\to \infty}\frac{1}{n}\sum_{j=0}^{n-1}\theta_j(b(x))
\ne \lim_{n\to \infty}\frac{1}{n}\sum_{j=0}^{n-1}\theta_j(l(x))
\end{equation*}
for almost every $x\in I_Q$, so that almost every $x\in I_Q$ has different greedy and lazy expansions.
Consequently, the set of numbers having unique expansions is a Lebesgue null set.

The difference of the integrals in \eqref{51} and \eqref{52} follows from the inequalities
\begin{align*}
\int_{I_Q}xh_{g,Q}(x)\,dx
&<(q_1-1)\Bigl(\int_{I_Q}x\,dx\Bigr)\Bigl(\int_{I_Q}h_{g,Q}(x)\,dx\Bigr)=\frac{1}{2(q_1-1)}
\intertext{and}
\int_{I_Q}xh_{l,Q}(x)\,dx
&>(q_1-1)\Bigl(\int_{I_Q}x\,dx\Bigr)\Bigl(\int_{I_Q}h_{l,Q}(x)\,dx\Bigr)=\frac{1}{2(q_1-1)},
\end{align*}
The corresponding weak inequalities follow from Lemma \ref{l51} because the identity function and $h_{l,Q}$ are both non-decreasing, while $h_{g,Q}$ is non-increasing on $I_Q$, and because non of the three functions is constant.
Indeed, the monotonicity statements follow from Theorem \ref{t11}, and the functions $h_{l,Q}$ and $h_{g,Q}$ are not constant because they have nonzero integral, and they vanish on non-empty intervals by Theorem \ref{t12}.
\end{proof}Finally, we deduce Theorem \ref{t15} from Theorem \ref{t14}.

\begin{proof}[Proof of Theorem~\ref{t15}]
For convenience, we henceforth denote the lazy and greedy maps by $T_0$ and $T_1$, respectively. 
We also define $T_{w_1\cdots w_n}=T_{w_n}\circ\cdots\circ T_{w_1}$ for $w_1\ldots w_n\in\{0,1\}^n$ and $n=1,2,\ldots .$
Observe that if $(w_i)_{i\geq 1}$ is an expansion of a number $x\in I_Q$, then $(w_{n+i})=w_{n+1}w_{n+2}\cdots$ is an expansion of  $T_{w_1\cdots w_n}(x)$ for each $n\geq 1$.  

Since $T_0$ and $T_1$ are nonsingular, preimages of null sets are again null sets. 
Since $\mathcal{U}_Q$ has measure zero by Theorem \ref{t14}, the set
\begin{equation*}
\mathcal{W}_Q=\bigcup_{n=1}^{\infty}\set{x\in I_Q:\,T_{w_1\cdots w_n}(x)\in \mathcal{U}_Q\text{ for some }w_1\ldots w_n\in\set{0,1}}
\end{equation*}
is a null set. 

Applying Sidorov's bifurcation argument \cite{Sidorov2003b, Sidorov2003c,Sidorov2009} we obtain that every number $x\in I_Q\setminus\mathcal{W}_Q$ has $2^{\aleph_0}$ expansions.
See also Dajani and de Vries \cite{DajaniDeVries2007} for more details on Sidorov's  reasoning. 
%
%To give some details, first we observe that if $(a_i)_{i\geq 1}$ is an expansion of a number $x\in I_Q$, then $(a_{n+i})=a_{n+1}a_{n+2}\cdots$ is an expansion of  $T_{a_1\cdots a_n}(x)$ for each $n\geq 1$. 
%
%Let $(a_i)_{i\geq 1}$ be an expansion of an arbitrary number $x\in I_Q\setminus\mathcal{W}_Q$, and
%\begin{equation*}
%r_1=\min\{n\geq 1:\,\sum_{l=1}^{\infty}\frac{a_{n+l-1}}{q_{a_n}\cdots q_{a_{n+l-1}}}\in I_Q\setminus\mathcal{V}_Q \}
%\end{equation*}
%\begin{equation*}
%    r_i=\min\{n\geq r_{i-1}:\,\sum_{l=1}^{\infty}\frac{a_{n+l-1}}{q_{a_n}\cdots q_{a_{n+l-1}}}\in I_Q\setminus\mathcal{V}_Q \}
%\end{equation*}
%then, proceeding recursively, each $a_{r_i}$ can be chosen 0 or 1. we can construct recursively a full binary tree of possible expansions of $x$, whence $\card\mathscr{E}_Q(x)=2^{\#\{r_1,r_2,\ldots\}}=2^{\aleph_0}$ for all $x\in I_Q\setminus\mathcal{V}_Q$.
\end{proof}

\color{black}

\end{document}